\documentclass[11pt]{article}

\usepackage{amssymb,amsmath,amsfonts,amsthm}
\usepackage{latexsym}
\usepackage{graphics}
\usepackage{indentfirst}

\newtheorem{innercustomthm}{Theorem}
\newenvironment{customthm}[1]
  {\renewcommand\theinnercustomthm{#1}\innercustomthm}
  {\endinnercustomthm}

\newtheorem{innercustompro}{Proposition}
\newenvironment{custompro}[1]
  {\renewcommand\theinnercustompro{#1}\innercustompro}
  {\endinnercustompro}
\newtheorem*{thm*}{Theorem}
\newtheorem{thm}{Theorem}
\newtheorem{lem}[thm]{Lemma}
\newtheorem{pro}[thm]{Proposition}
\newtheorem{obs}[thm]{Observation}
\newtheorem{cor}[thm]{Corollary}

\newtheorem{ques}[thm]{Question}

\newcommand{\N}{\mathbb{N}}

\newcommand\set[1]{\left\{ #1 \right\}}
\newcommand\floor[1]{\left\lfloor #1 \right\rfloor}
\newcommand\ceil[1]{\left\lceil #1 \right\rceil}

\begin{document}

\title{Proportional 2-Choosability with a Bounded Palette}

\author{Jeffrey A. Mudrock\footnotemark[1], Robert Piechota\footnotemark[1], Paul Shin\footnotemark[1], and Tim Wagstrom\footnotemark[1]}

\footnotetext[1]{Department of Mathematics, College of Lake County, Grayslake, IL 60030.  E-mail:  {\tt {jmudrock@clcillinois.edu}}}

\maketitle

\begin{abstract}

Proportional choosability is a list coloring analogue of equitable coloring.  Specifically, a $k$-assignment $L$ for a graph $G$ specifies a list $L(v)$ of $k$ available colors to each $v \in V(G)$.  An $L$-coloring assigns a color to each vertex $v$ from its list $L(v)$.  A \emph{proportional $L$-coloring} of $G$ is a proper $L$-coloring in which each color $c \in \bigcup_{v \in V(G)} L(v)$ is used $\lfloor \eta(c)/k \rfloor$ or $\lceil \eta(c)/k \rceil$ times where $\eta(c)=\left\lvert{\{v \in V(G) : c \in L(v) \}}\right\rvert$.  A graph $G$ is \emph{proportionally $k$-choosable} if a proportional $L$-coloring of $G$ exists whenever $L$ is a $k$-assignment for $G$.  Motivated by earlier work, we initiate the study of proportional choosability with a bounded palette by studying proportional 2-choosability with a bounded palette.  In particular, when $\ell \geq 2$, a graph $G$ is said to be \emph{proportionally $(2, \ell)$-choosable} if a proportional $L$-coloring of $G$ exists whenever $L$ is a $2$-assignment for $G$ satisfying $|\bigcup_{v \in V(G)} L(v)| \leq \ell$.  We observe that a graph is proportionally $(2,2)$-choosable if and only if it is equitably 2-colorable.  As $\ell$ gets larger, the set of proportionally $(2, \ell)$-choosable graphs gets smaller.  We show that whenever $\ell \geq 5$ a graph is proportionally $(2, \ell)$-choosable if and only if it is proportionally 2-choosable.  We also completely characterize the connected proportionally $(2, \ell)$-choosable graphs when $\ell = 3,4$.  

\medskip



\end{abstract}

\section{Introduction}\label{intro}

In this paper all graphs are nonempty, finite, simple graphs unless otherwise noted.  Generally speaking we follow West~\cite{W01} for terminology and notation.  The set of natural numbers is $\N = \{1,2,3, \ldots \}$.  For $m \in \N$, we write $[m]$ for the set $\{1, \ldots, m \}$.  If $G$ is a graph and $S \subseteq V(G)$, we use $G[S]$ for the subgraph of $G$ induced by $S$.  We write  $\Delta(G)$ for the maximum degree of a vertex in $G$.  We write $K_{n,m}$ for the complete bipartite graph with partite sets of size $n$ and $m$.  When $G$ is a path on $n$ vertices, $V(G)= \{v_1, \ldots, v_n \}$, and two vertices are adjacent in $G$ if and only if they appear consecutively in the ordering: $v_1, \ldots, v_n$, we say the vertices are written \emph{in order} when we write $v_1, \ldots, v_n$.  When $C$ is a cycle on $n$ vertices ($n \geq 3$ since $C$ is simple), $V(C)= \{v_1, \ldots, v_n \}$, and $E(C) = \{\{v_1,v_2 \}, \{v_2, v_3 \}, \ldots, \{v_{n-1}, v_n \}, \{v_n, v_1 \} \}$, then we say the vertices are written in \emph{cyclic order} when we write $v_1, \ldots, v_n$.  When $G_1$ and $G_2$ are vertex disjoint graphs, we write $G_1 + G_2$ for the disjoint union of $G_1$ and $G_2$.

In 2019 a new notion combining the notions of list coloring and equitable coloring called proportional choosability was introduced~\cite{KM19}.  In this paper, we study proportional choosability with a bounded palette.  We begin by briefly reviewing some important notions.  

\subsection{List Coloring with a Bounded Palette}

Given a graph $G$, in the classical vertex coloring problem we wish to color the elements of $V(G)$ with colors from the set $[k]$ so that adjacent vertices receive different colors, a so-called \emph{proper $k$-coloring}.  We say $G$ is \emph{$k$-colorable} when a proper $k$-coloring of $G$ exists.  The \emph{chromatic number} of $G$, denoted $\chi(G)$, is the smallest $k$ such that $G$ is $k$-colorable.  

List coloring is a variation on classical vertex coloring, and it was introduced independently by Vizing~\cite{V76} and Erd\H{o}s, Rubin, and Taylor~\cite{ET79} in the 1970's.  For list coloring, we associate with a graph $G$ a \emph{list assignment} $L$ that assigns to each vertex $v \in V(G)$ a list $L(v)$ of available colors.  We say $G$ is \emph{$L$-colorable} if there exists a proper coloring $f$ of $G$ such that $f(v) \in L(v)$ for each $v \in V(G)$ (we refer to $f$ as a \emph{proper $L$-coloring} of $G$).  A list assignment $L$ is called a \emph{k-assignment} for $G$ if $|L(v)|=k$ for each $v \in V(G)$.  We say $G$ is \emph{k-choosable} if $G$ is $L$-colorable whenever $L$ is a $k$-assignment for $G$.  

The study of list coloring with a bounded palette began in 2005~\cite{KS05}.  Suppose that $L$ is a list assignment for a graph $G$.  The \emph{palette of colors associated with $L$} is $\cup_{v \in V(G)} L(v)$.  From this point forward, we use $\mathcal{L}$ to denote the palette of colors associated with $L$ whenever $L$ is a list assignment.  Suppose $1 \leq k \leq \ell$.  A list assignment $L$ for a graph $G$ is a \emph{$(k, \ell)$-assignment} for $G$ if $L$ is a $k$-assignment for $G$ and $\mathcal{L} \subseteq [\ell]$.  Notice that if $L$ is a $(k, \ell)$-assignment for $G$, we can view $L$ as a function with domain $V(G)$ and codomain equal to the set of $k$-element subsets of $[\ell]$.  We say $G$ is \emph{$(k, \ell)$-choosable} if $G$ is $L$-colorable whenever $L$ is a $(k, \ell)$-assignment for $G$.  Clearly, a graph is $(k,k)$-choosable if and only if it is $k$-colorable.  In~\cite{DW17} the complexity of $(k, \ell)$-choosability is studied for grids (i.e. the Cartesian product of two paths), subgrids (i.e. induced subgraphs of grids), 3-colorable planar graphs, and triangle-free planar graphs.

In~\cite{KS05} it is shown that for any $k \geq 2$, there is a $C \in \N$ satisfying $C = O(k 16^k \ln k)$ as $k \rightarrow \infty$ such that if $G$ is $(k, 2k-1)$-choosable, then $G$ is $C$-choosable.  In 2015, it was subsequently demonstrated that this constant $C$ must also satisfy $C = \Omega(4^k/\sqrt{k})$ as $k \rightarrow \infty$ (see~\cite{BK15}).  Importantly, results like this show that understanding list coloring with a bounded palette can provide us with information about list coloring in general.  On the other hand, graphs that fail to be $k$-choosable can be $(k , \ell)$-choosable.  Indeed, for each $k$ and $\ell$ satisfying $3 \leq k \leq \ell$, there is a graph $G$ that is $(k, \ell)$-choosable but not $(k, \ell+1)$-choosable (see~\cite{KS05}). 

\subsection{Equitable Coloring and Proportional Choosability}

\subsubsection{Equitable Coloring}

Equitable coloring is another variation on the classical vertex coloring problem that began with a conjecture of Erd\H{o}s in 1964~\cite{E64}.  Equitable coloring was formally defined by Meyer in 1973~\cite{M73}.  Specifically, an \emph{equitable $k$-coloring} of a graph $G$ is a proper $k$-coloring $f$ of $G$ such that the sizes of the color classes differ by at most one (where a proper $k$-coloring has exactly $k$ color classes).  In an equitable $k$-coloring, the color classes associated with the coloring are each of size $\lceil |V(G)|/k \rceil$ or $\lfloor |V(G)|/k \rfloor$.  We say that a graph $G$ is \emph{equitably $k$-colorable} if there exists an equitable $k$-coloring of $G$.  Equitable coloring has been applied in various contexts (for example, see~\cite{JR02, KJ06, P01, T73}).  Furthermore,   in 1970 Hajn\'{a}l and Szemer\'{e}di~\cite{HS70} proved the 1964 conjecture of Erd\H{o}s: every graph $G$ has an equitable $k$-coloring when $k \geq \Delta(G)+1$. 

Unlike classical vertex coloring, increasing the number of colors can make equitable coloring more difficult.  For example, for any $m \in \N$, $K_{2m+1, 2m+1}$ is equitably $2m$-colorable, but it is not equitably $(2m+1)$-colorable.  Moreover, unlike classical vertex coloring, the property of being equitably $k$-colorable is not monotone.  For example, $K_{3,3}$ is equitably 2-colorable, but $K_{1,3}$ is not equitably 2-colorable.  

\subsubsection{Proportional Choosability}

In 2003, Kostochka, Pelsmajer, and West~\cite{KP03} introduced a list version of equitable coloring called equitable choosability which has received quite a bit of attention in the literature.  If $L$ is a $k$-assignment for the graph $G$, a proper $L$-coloring of $G$ is an \emph{equitable $L$-coloring} of $G$ if each color in $\mathcal{L}$ appears on at most $\lceil |V(G)|/k \rceil$ vertices.  We say $G$ is \emph{equitably $k$-choosable} if an equitable $L$-coloring of $G$ exists whenever $L$ is a $k$-assignment for $G$.  While equitable choosability is a useful notion in many contexts, it does not place a lower bound on how many times a color must be used, whereas in an equitable $k$-coloring of $G$ each color must be used at least $\lfloor |V(G)|/k \rfloor$ times.

Kaul, Pelsmajer, Reiniger, and the first author~\cite{KM19} introduced a new list analogue of equitable coloring called proportional choosability which places both an upper and lower bound on how many times a color must be used in a list coloring.  Specifically, suppose that $L$ is a $k$-assignment for a graph $G$.  For each color $c \in \mathcal{L}$, the \emph{multiplicity of $c$ in $L$} is the number of vertices $v$ whose list $L(v)$ contains $c$.  The multiplicity of $c$ in $L$ is denoted by $\eta_L(c)$ (or simply $\eta(c)$ when the list assignment is clear).  So, $\eta_L(c)=\left\lvert{\{v \in V(G) : c \in L(v) \}}\right\rvert$.  A proper $L$-coloring $f$ for $G$ is a \emph{proportional $L$-coloring} of $G$ if for each $c \in \mathcal{L}$, $f^{-1}(c)$, the color class of $c$, is of size
$$ \left \lfloor \frac{\eta(c)}{k} \right \rfloor \; \; \text{or} \; \; \left \lceil \frac{\eta(c)}{k} \right \rceil.$$
We say that $G$ is \emph{proportionally $L$-colorable} if a proportional $L$-coloring of $G$ exists, and we say $G$ is \emph{proportionally $k$-choosable} if $G$ is proportionally $L$-colorable whenever $L$ is a $k$-assignment for $G$.  Proportional choosability has some beautiful properties, some of which, at first glance, may seem quite surprising.

\begin{pro} [\cite{KM19}] \label{pro: motivation}
If $G$ is proportionally $k$-choosable, then $G$ is both equitably $k$-choosable and equitably $k$-colorable.
\end{pro}

\begin{pro} [\cite{KM19}] \label{pro: monoink}
If $G$ is proportionally $k$-choosable, then $G$ is proportionally $(k+1)$-choosable.
\end{pro}

\begin{pro} [\cite{KM19}] \label{lem: monotone}
Suppose $H$ is a subgraph of $G$.  If $G$ is proportionally $k$-choosable, then $H$ is proportionally $k$-choosable.
\end{pro}

Notice that Propositions~\ref{pro: monoink} and~\ref{lem: monotone} are particularly interesting since they do not hold in the contexts of equitable coloring and equitable choosability.  Recently, a nice characterization of the proportionally 2-choosable graphs was discovered; this characterization inspired the questions that lead to this paper.  Recall that a \emph{linear forest} is a disjoint union of paths.

\begin{thm} [\cite{KM20, M18}] \label{thm: full2character}
A graph $G$ is proportionally 2-choosable if and only if $G$ is a linear forest such that the largest component of $G$ has at most five vertices and all the other components of $G$ have two or fewer vertices.
\end{thm}  

\subsection{Proportional Choosability with a Bounded Palette}

Having defined proportional choosability, it is natural to consider proportional choosability with a bounded palette.  Suppose $1 \leq k \leq \ell$.  We say a graph $G$ is \emph{proportionally $(k,\ell)$-choosable} if $G$ is proportionally $L$-colorable whenever $L$ is a $(k, \ell)$-assignment for $G$.  Two properties of proportional $(k, \ell)$-choosability are easy to immediately prove.

\begin{pro} \label{pro: k,k}
For each $k \in \N$, $G$ is proportionally $(k,k)$-choosable if and only if $G$ is equitably $k$-colorable.
\end{pro}

\begin{proof}
Suppose $G$ is proportionally $(k,k)$-choosable. Let $L$ be a $k$-assignment for $G$ such that $L(v)= [k]$ for all $v \in V(G)$. Note that $\eta (1) = \cdots = \eta (k) = |V(G)|$. Since $L$ is a $(k,k)$-assignment for $G$, we know there is a proportional $L$-coloring $f$ of $G$.  Clearly, $f$ is also an equitable $k$-coloring of $G$. 

Conversely, suppose $G$ is equitably $k$-colorable and $L$ is an arbitrary $(k,k)$-assignment for $G$. Notice that an equitable $k$-coloring of $G$ exists, and $L(v) = [k]$ for each $v \in V(G)$. The result follows since an equitable $k$-coloring of $G$ is also a proportional $L$-coloring of $G$.
\end{proof}

\begin{pro} \label{pro: monotonicity}
Suppose $1 \leq k \leq \ell$.  If $G$ is proportionally $(k, \ell + 1)$-choosable, then $G$ is proportionally $(k, \ell)$-choosable.
\end{pro}

\begin{proof}
Suppose $G$ is proportionally $(k,\ell+1)$-choosable, and suppose $L$ is an arbitrary $(k,\ell)$-assignment for $G$.  Clearly, $L$ is also a $(k, \ell+1)$-assignment for $G$.  Since $G$ is proportionally $(k,\ell+1)$-choosable, we know that $G$ is proportionally $L$-colorable.
\end{proof}

The following question lead to the results in this paper.

\begin{ques} \label{ques: main}
For each $\ell \geq 2$, what graphs are proportionally $(2, \ell)$-choosable?
\end{ques}

Suppose $\mathcal{G}$ is the set of proportionally $2$-choosable graphs.  Notice that if $i \geq 2$ and $\mathcal{G}_i$ is the set of graphs that are proportionally $(2, i)$-choosable, then by Proposition~\ref{pro: monotonicity},
$$ \mathcal{G}_2 \supseteq \mathcal{G}_3 \supseteq \mathcal{G}_4 \supseteq \cdots .$$
By Theorem~\ref{thm: full2character}, for every $\ell \in \N$, $\mathcal{G}_{\ell}$ contains all linear forests such that the largest component has at most five vertices and all the other components have two or fewer vertices (i.e. $\mathcal{G}$ is a subset of $\mathcal{G}_{\ell}$ for each $\ell \in \N$).  Furthermore, Proposition~\ref{pro: k,k} tells us that $\mathcal{G}_2$ is exactly the set of equitably $2$-colorable graphs.  Since an $n$-vertex graph is proportionally $k$-choosable if and only if it is proportionally $(k,kn)$-choosable the following question and its generalization are natural. 

\begin{ques} \label{ques: smallest}
Is there a constant $\mu$ such that any graph $G$ is proportionally 2-choosable if and only if $G$ is proportionally $(2, \mu)$-choosable?
\end{ques}    

\begin{ques} \label{ques: smallestgen}
For each $k \geq 2$, is there a constant $\mu_k$ such that any graph $G$ is proportionally $k$-choosable if and only if $G$ is proportionally $(k , \mu_k)$-choosable? 
\end{ques}

Question~\ref{ques: smallestgen} is open for each $k \geq 3$.  The answer to Question~\ref{ques: smallest} is yes, and interestingly, the smallest such $\mu$ for which the answer is yes is 5.  Specifically, using the notation above, we will see below that
$$\mathcal{G}_2 \supset \mathcal{G}_3 \supset \mathcal{G}_4 \supset \mathcal{G}_5 \; \; \text{and} \; \; \mathcal{G}_\ell = \mathcal{G} \; \; \text{for each} \; \; \ell \geq 5.$$ 

\subsection{Outline of Results and an Open Question}

We will answer Question~\ref{ques: main} for $\ell = 2$ and each $\ell \geq 5$ which will give us an answer to Question~\ref{ques: smallest}.  We give a partial answer to Question~\ref{ques: main} when $\ell = 3, 4$.  The proofs of many of our results rely on finding ways to extend Proposition~\ref{lem: monotone} to the bounded palette context.  This presents some difficulties as the proof of Proposition~\ref{lem: monotone} relies on the construction of a list assignment that may have a large palette size (cf. the proof of Proposition~21 in~\cite{KM19}).  The characterization of the proportionally $(2,2)$-choosable graphs below follows immediately from Proposition~\ref{pro: k,k} and a simple characterization of equitably 2-colorable graphs.

\begin{obs} \label{thm: lis2}
A graph $G$ is proportionally $(2,2)$-choosable if and only if $G$ is a bipartite graph with a bipartition $X, Y$ satisfying $||X|-|Y|| \leq 1$.
\end{obs}

Our next result answers Question~\ref{ques: main} for each $\ell \geq 5$.

\begin{thm} \label{thm: lis5+}
For each $\ell \geq 5$, a graph $G$ is proportionally $(2, \ell)$-choosable if and only if $G$ is a linear forest such that the largest component of $G$ has at most $5$ vertices and all other components of $G$ have at most $2$ vertices.
\end{thm}

When it comes to proportional $(2,4)$-choosability and proportional $(2,3)$-choosability, we have characterizations for connected graphs.

\begin{thm} \label{thm: lis4}
A connected graph $G$ is proportionally $(2, 4)$-choosable if and only if $G = P_n$ where $n \leq 5$ or $n = 7$.
\end{thm}

\begin{thm} \label{thm: lis3}
A connected graph $G$ is proportionally $(2,3)$-choosable if and only if $G= P_n$ for some $n \in \N$.
\end{thm}

Theorem~\ref{thm: lis3} is particularly interesting since in general, little is known about the proportional choosability of paths (cf. Questions~6 and~7 in~\cite{KM20}).  With Theorems~\ref{thm: lis4} and~\ref{thm: lis3} in mind, the following question is natural.

\begin{ques} \label{ques: open}
For $\ell=3,4$ what graphs are proportionally $(2, \ell)$-choosable?
\end{ques}

One might conjecture that a graph $G$ is proportionally $(2,4)$-choosable (resp. $(2,3)$-choosable) if and only if the components of $G$ are proportionally $(2,4)$-choosable (resp. $(2,3)$-choosable).  This conjecture however is not correct in both directions for proportional $(2,4)$-choosability, and the ``only if" direction of this conjecture is not correct for proportional $(2,3)$-choosability.  The following results demonstrate this.  When reading the results below note that by Theorems~\ref{thm: lis4} and~\ref{thm: lis3} we know: $P_3$ is proportionally $(2,4)$-choosable, $P_6$ is not proportionally $(2,4)$-choosable, and $C_4$ is not proportionally $(2,3)$-choosable. 

\begin{pro} \label{pro: P6}
$P_3+P_3$ is not proportionally $(2,4)$-choosable.
\end{pro}

\begin{pro} \label{pro: lis4}
$P_6 + P_1$ is proportionally $(2,4)$-choosable.
\end{pro}

\begin{pro} \label{pro: lis3}
$C_4 + P_1$ is proportionally $(2,3)$-choosable.
\end{pro}

\section{Theorem~\ref{thm: lis5+}}

We begin by proving three lemmas.

\begin{lem} \label{pro: degree}
If $G$ contains a copy of $K_{1,3}$ as a subgraph, then $G$ is not proportionally $(2,3)$-choosable.  Consequently, if a graph $G$ is proportionally $(2, \ell)$-choosable for some $\ell \geq 3$, then $\Delta(G) \leq 2$.
\end{lem}

\begin{proof}
Suppose $H$ is a subgraph of $G$ such that $H = K_{1,3}$, and suppose $H$ has bipartition $\set{a}$ and $\set{b_{1},b_{2},b_{3}}$.  To prove the desired, we will construct a $(2,3)$-assignment, $L$, for $G$ such that there is no proportional $L$-coloring of $G$. Suppose $L$ is the $(2,3)$-assignment for $G$ such that for each $v \in V(H)$, $L(v)=\set{1,2}$, and for each $v \in V(G)-V(H)$, $L(v)=\set{2,3}$.  For the sake of contradiction, suppose that $f$ is a proportional $L$-coloring of $G$.  Note that $\eta(1)=4$, so $|f^{-1}(1)|=2$.  Clearly, $f(a)=1$ or $f(a)=2$.  If $f(a)=1$, then $f(b_{i})=2$ for each $i\in[3]$, and $|f^{-1}(1)|=1$.  If $f(a)=2$, then $f(b_{i})=1$ for each $i\in[3]$, and $|f^{-1}(1)|=3$.  In either case we have a contradiction.
\end{proof}  

\begin{lem} \label{pro: cycle in G}
If a graph contains a cycle, then it is not proportionally $(2, \ell)$-choosable for each $\ell \geq 4$.
\end{lem}

\begin{proof}
Suppose $G$ is an arbitrary graph that contains a cycle $C$. By Proposition~\ref{pro: monotonicity}, it suffices to show that $G$ is not proportionally $(2,4)$-choosable. If $C$ is an odd cycle, then $G$ is not $2$-colorable; thus, $G$ is not proportionally $(2,4)$-choosable. So, we may suppose that $C$ is an even cycle.

Suppose the vertices of $C$ written in cyclic order are: $v_1, \ldots, v_{2k+2}$ where $k \in \N$.  We will now construct a $(2,4)$-assignment, $L$, for $G$ such that there is no proportional $L$-coloring of $G$. Suppose $L$ is the $(2,4)$-assignment for $G$ given by $L(v_{2i-1}) = \{1, 2\}$ and $L(v_{2i}) = \{1, 3\}$ for each $i \in [k+1]$, and $L(v) = \{3, 4\}$ if $v \in V(G) - V(C)$. Notice that $\eta(1) = 2k+2$ and $\eta(2) = k+1$.  For the sake of contradiction, suppose $f$ is a proportional $L$-coloring of $G$. This implies that $|f^{-1}(1)| = k+1$ and 
$$0 < \left \lfloor \frac{k+1}{2} \right \rfloor \leq |f^{-1}(2)| \leq \left \lceil \frac{k+1}{2} \right \rceil < k+1.$$ 
Since $C$ contains exactly two independent sets of size at least $k+1$ and $1 \notin L(v)$ for each $v \in V(G) - V(C)$, either $f(v_{2i}) = 1$ for each $i \in [k+1]$ or $f(v_{2i-1}) = 1$ for each $i \in [k+1]$. This implies that $|f^{-1}(2)| = k+1$ or $|f^{-1}(2)| = 0$ which in either case is a contradiction.
\end{proof}

\begin{lem} \label{pro: 2K12-5}
If a graph contains a copy of $K_{1,2}+K_{1,2}$, then it is not proportionally $(2,\ell)$-choosable for each $\ell \geq 5$.
\end{lem}

\begin{proof}
Suppose $G$ is a graph that contains two vertex disjoint graphs $H_1$ and $H_2$ that are copies of $K_{1,2}$. By Proposition~\ref{pro: monotonicity}, it suffices to show that $G$ is not proportionally $(2,5)$-choosable. Suppose $H_1$ has bipartition $A_1, B_1$, where $A_1 = \{a_1\}$ and $B_1 = \{b_0, b_1\}$. Suppose $H_2$ has bipartition $A_2, B_2$, where $A_2 = \{a_2\}$ and $B_2 = \{b_2, b_3\}$. We will now construct a $(2,5)$-assignment $L$ for $G$ such that there is no proportional $L$-coloring of $G$.  Suppose $L$ is the $(2,5)$-assignment for $G$ given by $L(a_1) = L(a_2) = \{1, 2\}$, $L(b_0) = L(b_1) = \{1, 3\}$, $L(b_2) = L(b_3) = \{1, 4\}$, and $L(v) = \{1, 5\}$ if $v \in V(G) - V(H_1 + H_2)$. Notice that $\eta(i) = 2$ for $i = 2,3,4$.

For the sake of contradiction, suppose $f$ is a proportional $L$-coloring of $G$. This means that $|f^{-1}(i)| = 1$ for $i = 2,3,4$. Thus, $f(a_1) = 1$ or $f(a_2) = 1$. This implies that $|f^{-1}(3)| = 2$ or $|f^{-1}(4)| = 2$ respectively which in either case is a contradiction.
\end{proof}

We are now ready to prove Theorem~\ref{thm: lis5+} which we restate.

\begin{customthm}{\bf \ref{thm: lis5+}}
For each $\ell \geq 5$, a graph $G$ is proportionally $(2, \ell)$-choosable if and only if $G$ is a linear forest such that the largest component of $G$ has at most $5$ vertices and all other components of $G$ have at most $2$ vertices.
\end{customthm}

\begin{proof}
Throughout the proof, suppose $\ell$ is a fixed natural number satisfying $\ell \geq 5$.  Suppose that $G$ is a linear forest such that the largest component of $G$ has at most 5 vertices and all other components of $G$ have at most 2 vertices. By Theorem~\ref{thm: full2character}, we know $G$ is proportionally $(2,\ell)$-choosable. 

Conversely, suppose that $G$ is proportionally $(2,\ell)$-choosable. By Lemma~\ref{pro: degree} we know that $\Delta (G) \leq 2$, and by Lemma~\ref{pro: cycle in G} we know that $G$ can not contain a cycle. This means that $G$ must be a linear forest. Finally, by Lemma~\ref{pro: 2K12-5} we know that $G$ can not contain a copy of $K_{1,2} + K_{1,2}$ (i.e. $P_3 + P_3$). Thus, $G$ must be a linear forest such that the longest path has at most 5 vertices and all other paths have at most 2 vertices.
\end{proof}

\section{Theorem~\ref{thm: lis4}}

In this section we prove Theorem~\ref{thm: lis4} which we restate.

\begin{customthm}{\bf \ref{thm: lis4}}
A connected graph $G$ is proportionally $(2, 4)$-choosable if and only if $G = P_n$ where $n \leq 5$ or $n = 7$.
\end{customthm}

With the exception of $P_7$, which we will take care of near the end of this section, note that the ``if" direction of Theorem~\ref{thm: lis4} is implied by Theorem~\ref{thm: full2character}.  Conversely, by Lemmas~\ref{pro: degree} and~\ref{pro: cycle in G}, we know that if $G$ is a connected graph that is proportionally $(2,4)$-choosable, $G$ must be a path.  We now focus upon further narrowing down which paths are proportionally $(2,4)$-choosable.  We begin with two useful general results.

\begin{pro} \label{pro: span}
Suppose $1 \leq k \leq \ell$.  Suppose $G$ is a graph and $H$ is a spanning subgraph of $G$. If $H$ is not proportionally $(k, \ell)$-choosable, then $G$ is not proportionally $(k, \ell)$-choosable.
\end{pro}

\begin{proof}
Suppose $H$ is not proportionally $(k, \ell)$-choosable. For the sake of contradiction, assume $G$ is proportionally $(k, \ell)$-choosable. Suppose $L$ is an arbitrary $(k,\ell)$-assignment for $H$. Since $V(H) = V(G)$, we have that $L$ is also a $(k,\ell)$-assignment for $G$. It follows that there exists a proportional $L$-coloring, $f$, of $G$. Since $E(H) \subseteq E(G)$, we have that $f$ is also a proportional $L$-coloring of $H$. However, this implies that $H$ is proportionally $(k, \ell)$-choosable, and we have a contradiction.
\end{proof}

\begin{pro} \label{pro: addKk}
Suppose $1 \leq k \leq \ell$.  If $G$ is not proportionally $(k,\ell)$-choosable, then $G+K_k$ is not proportionally $(k,\ell)$-choosable.
\end{pro}

\begin{proof}
Let $H$ be a copy of $K_k$ with vertices $w_1, \ldots, w_k$,
and let $L'$ be a $(k,\ell)$-assignment for $G$ such that there is no proportional $L'$-coloring of $G$.
We construct a $(k,\ell)$-assignment, $L$, for $G+H$ as follows:
\[
L(v)=
\begin{cases}
L'(v)	& \text{if $v\in V(G)$}	\\
\set{1,2, \cdots , k} & \text{if $v\in V(H)$.}
\end{cases}
\]
To prove the desired, we will show there is no proportional $L$-coloring of $G+H$.  For the sake of contradiction, suppose $f$ is a proportional $L$-coloring of $G+H$. Without loss of generality, we may assume that $f(w_c) = c$ for each $ c \in [k]$.
So, $|f^{-1}(c) \cap V(H)|=1$ for each $c\in [k]$. We know $\floor{\eta_{L}(c)/k} \leq |f^{-1}(c)| \leq \ceil{\eta_{L}(c)/k}$ for each $c\in [k]$.  Since $\floor{\eta_{L}(c)/k} = \floor{\eta_{L'}(c)/k} + 1$ and $\ceil{\eta_{L}(c)/k} = \ceil{\eta_{L'}(c)/k} + 1$, restricting the domain of $f$ to $V(G)$ yields a proportional $L'$-coloring of $G$.  This however is a contradiction.
\end{proof}

The following Corollary immediately follows from Proposition~\ref{pro: addKk}.

\begin{cor} \label{cor: addp2}
If $G$ is not proportionally $(2,\ell)$-choosable where $\ell \geq 2$, then $G+P_2$ is not proportionally $(2,\ell)$-choosable.
\end{cor}

When it comes to proportional $(2,4)$-choosability, notice that combining Proposition~\ref{pro: span} with Corollary~\ref{cor: addp2} can allow us to prove that many paths are not proportionally $(2,4)$-choosable.  For example, if we know that $P_6$ is not proportionally $(2,4)$-choosable, then Proposition~\ref{pro: span} and Corollary~\ref{cor: addp2} can be used in an easy inductive argument to show that $P_n$ is not proportionally $(2,4)$-choosable whenever $n$ is even and at least 6.  We will now apply this idea to show that $P_n$ is not proportionally $(2,4)$-choosable when $n = 6$ or $n \geq 8$. 

\begin{custompro}{\bf \ref{pro: P6}}
$P_3+P_3$ is not proportionally $(2,4)$-choosable.
\end{custompro}

\begin{proof}
Notice that $P_3 = K_{1,2}$.  Suppose $G_1$ and $G_2$ are vertex disjoint copies of $K_{1,2}$. Suppose $G_1$ has bipartition $\set{v_1}$, $\set{w_1,w_2}$, and suppose $G_2$ has bipartition $\set{v_2}$, $\set{w_3,w_4}$.  Let $G=G_1+G_2$, and let $L$ be the $(2,4)$-assignment for $G$ given by:
$L(v_1)=L(v_2)=\set{1,2}$, $L(w_{1})=L(w_{2})=\set{1,3}$, $L(w_{3})=L(w_{4})=\set{1,4}$.  To prove the desired, we will show there is no proportional $L$-coloring of $G$.  For the sake of contradiction, suppose $f$ is a proportional $L$-coloring of $G$. Since $\eta(2)=2$, either $f(v_1)=2$ or $f(v_2)=2$.
Note that if $f(v_1)=2$, then $f(v_2)=1$ and $f$ uses 4 too many times.  Similarly, if $f(v_2)=2$, then $f(v_1)=1$ and $f$ uses 3 too many times. So, we have a contradiction.
\end{proof}

\begin{cor} \label{cor: basecase}
$P_6$ and $P_8$ are not proportionally $(2,4)$-choosable.
\end{cor}

\begin{proof}
Since $P_3+P_3$ is a spanning subgraph of $P_6$, Proposition~\ref{pro: span} implies $P_6$ is not proportionally $(2,4)$-choosable.
By Corollary~\ref{cor: addp2}, we know $P_6+P_2$ is not proportionally $(2,4)$-choosable.  Proposition~\ref{pro: span} then implies $P_8$ is also not proportionally $(2,4)$-choosable.
\end{proof}

\begin{pro} \label{pro: P9}
$P_9$ is not proportionally $(2,4)$-choosable.
\end{pro}

\begin{proof}
Suppose $G=P_9$.  Suppose the vertices of $G$ in order are: $v_1, v_2, \ldots , v_9$. To prove the desired, we will construct a $(2,4)$-assignment, $L$, for $G$ with the property that there is no proportional $L$-coloring of $G$.  Let $L(v_1)=L(v_3)=L(v_5)=\{1,2\}$, $L(v_2)=L(v_4)=L(v_7)=\{1,3\}$, $L(v_6)=L(v_8)=\{1,4\}$, and $L(v_9)=\{2,3\}$.  We claim that $G$ is not proportionally $L$-colorable. Suppose for the sake of contradiction that $f$ is a proportional $L$-coloring of $G$. Note that since $\eta(4)=2$ it must be that either $f(v_6) =4$ and $f(v_8)=1$ or $f(v_8) =4$ and $f(v_6)=1$. So, it must be that $f(v_7) =3$. Note that $\eta(1) = 8$ which means that $|f^{-1}(1)|=4$. If $f(v_6) =1$, then $|f^{-1}(1)| < 4$ which is a contradiction. So, $f(v_8) = 1$ and $f(v_6)=4$ which means that $f(v_1) = f(v_3)=f(v_5) = 1$.  This implies $f(v_2) = f(v_4) =3$ which means $|f^{-1}(3)| \geq 3$. This however contradicts the fact that $|f^{-1}(3)|=2$ since $\eta(3) = 4$.
\end{proof}

\begin{lem} \label{pro: induction}
$P_n$ is not proportionally $(2,4)$-choosable for each $n\geq 8$.
\end{lem}

\begin{proof}
Let $G=P_n$ where $n \geq 8$.  We prove the desired by induction on $n$.  For the base case, notice that the result follows when $n=8,9$ by Corollary~\ref{cor: basecase} and Proposition~\ref{pro: P9}.

Now, suppose $n\geq 10$ and the desired result holds for each integer greater than or equal to 8 and less than $n$.
We know $n-2\geq 8$, so $P_{n-2}$ is not proportionally $(2,4)$-choosable.  By Corollary~\ref{cor: addp2}, $P_{n-2}+P_2$ is not proportionally $(2,4)$-choosable.  Proposition~\ref{pro: span} then implies that $G$ is not proportionally $(2,4)$-choosable.
\end{proof}

We need one more result before we prove Theorem~\ref{thm: lis4}.  As mentioned at the start of this Section, we need to show $P_7$ is proportionally $(2,4)$-choosable which we will do with the aid of a computer.

\begin{pro} \label{pro: P7}
$P_7$ is proportionally $(2,4)$-choosable.
\end{pro}

\begin{proof}
Suppose $G = P_7$.  Clearly, there are $\binom{4}{2}^7 = 279936$ possible $(2,4)$-assignments for $G$.  Also, for each $(2,4)$-assignment, $L$, for $G$ there are $2^7 = 128$ possible $L$-colorings of $G$ (some of which may not be proper).  We used a Python program (see Appendix A for the code) that found at least one proportional coloring of $G$ for each of $279936$ possible $(2,4)$-assignments for $G$.  It immediately follows that $G$ is proportionally $(2,4)$-choosable.
\end{proof}

It should be noted that Propositions~\ref{pro: span} and~\ref{pro: P7} immediately imply that $P_6 + P_1$ is proportionally $(2,4)$-choosable (i.e. Proposition~\ref{pro: lis4}).  We are now ready to prove Theorem~\ref{thm: lis4}.

\begin{proof}
Suppose that $G = P_n$ where $n \leq 5$ or $n =7$. When $n \leq 5$ we know by Theorem~\ref{thm: full2character} that $G$ is proportionally $(2,4)$-choosable. When $n =7$ we know by Proposition~\ref{pro: P7} that $G$ is proportionally $(2,4)$-choosable. 

Conversely, suppose that $G$ is proportionally $(2,4)$-choosable. Since $G$ is connected, Lemmas~\ref{pro: degree} and~\ref{pro: cycle in G} imply that $G = P_n$ for some $n \in \N$. Finally, by Corollary~\ref{cor: basecase} and Lemma~\ref{pro: induction} we know that $n \leq 7$ and $n \neq 6$. Thus, $G = P_n$ where $n \leq 5$ or $n = 7$.
\end{proof}

\section{Theorem~\ref{thm: lis3}}

In this section we prove Theorem~\ref{thm: lis3} which we restate.

\begin{customthm}{\bf \ref{thm: lis3}}
A connected graph $G$ is proportionally $(2,3)$-choosable if and only if $G= P_n$ for some $n \in \N$.
\end{customthm}

Proving the ``if" direction will require a bit of effort.  So, we begin by concentrating on the ``only if" direction. By Lemma~\ref{pro: degree}, we know that if $G$ is a connected graph that is proportionally $(2,3)$-choosable, $G$ must be a cycle or path.  So, we need to address cycles.

\begin{lem} \label{pro: cycle is G}
If $G = C_n$ for some $n \geq 3$, then $G$ is not proportionally $(2, 3)$-choosable.
\end{lem}

\begin{proof}
Suppose the vertices of $G$ written in cyclic order are: $v_1, \ldots, v_n$. If $n$ is odd, then $G$ is not $2$-colorable; hence, $G$ is not proportionally $(2,3)$-choosable.  Now, suppose $n = 2k+2$ for $k \in \mathbb{N}$. To prove the desired, we will construct a $(2,3)$-assignment $L$ for $G$ such that there is no proportional $L$-coloring of $G$.  Let $L$ be the $(2,3)$-assignment for $G$ given by: $L(v_{2i}) = \{1, 2\}$ and $L(v_{2i - 1}) = \{1, 3\}$ for each $i \in [k+1]$. For the sake of contradiction, suppose $f$ is a proportional $L$-coloring of $G$. Notice that $\eta(1) = 2k+2$, so $|f^{-1}(1)| = k+1$; also, $G$ contains precisely two independent sets of size at least $k+1$.

Without loss of generality, suppose $f(v_{2i-1}) = 1$ for each $i \in [k+1]$. Notice that $\eta(3) = k+1$ implies that 
$$|f^{-1}(3)| \geq \left \lfloor \frac{k+1}{2} \right \rfloor > 0.$$ However, $3$ is not an element of $L(v_{2i})$ for all $i \in [k+1]$.  So, $|f^{-1}(3)| = 0$ which is a contradiction.
\end{proof}

We will now concentrate on proving that $P_n$ is proportionally $(2,3)$-choosable for each $n \in \N$.  We first make an observation: for any $2$-assignment $L$ for a graph $G$, the number of colors with odd multiplicity in $\mathcal{L}$ is even.  So, if $L$ is a $(2,3)$ assignment for $G$, either none of the colors in $[3]$ have odd multiplicity or exactly two of the colors in $[3]$ have odd multiplicity.  In order to prove Proposition~\ref{pro: evenmult} below, we will prove four Observations and three Lemmas within the body of the proof of Proposition~\ref{pro: evenmult}.

\begin{pro} \label{pro: evenmult}
Suppose $n \in \N$.  Suppose $G = P_n$ and $L$ is a $(2,3)$-assignment for $G$ where all the colors in $\mathcal{L}$ have even multiplicity.  Then, $G$ is proportionally $L$-colorable.
\end{pro}

\begin{proof}
We prove the result by induction on $n$.  Notice that the desired result holds for all $n \leq 5$ by Theorem~\ref{thm: full2character}.

So, suppose that $n \geq 6$ and the desired result holds for all natural numbers less than $n$.  For the sake of contradiction, suppose that $L$ is a $(2,3)$-assignment for $G$ where all the colors in $\mathcal{L}$ have even multiplicity, and suppose there is no proportional $L$-coloring of $G$.  Suppose that the vertices of $G$ in order are $v_1, v_2, \ldots, v_n$.  Also, suppose that $\eta(1)=2a_1$, $\eta(2)=2a_2$, and $\eta(3)=2a_3$.   The strategy of the proof is to determine as much as we can about what $L$ must look like and then show that a proportional $L$-coloring of $G$ must actually exist.

\textbf{Observation 1:} For each $i \in [n-1]$, it must be that $L(v_i) \neq L(v_{i+1})$.  To see why this is so, suppose that there is a $j \in [n-1]$ such that $L(v_j) = L(v_{j+1})$.  Without loss of generality, suppose $L(v_j)=\{1,2\}$.  Let $G'$ be the path on $n-2$ vertices obtained from $G$ by deleting $v_j$ and $v_{j+1}$, and then if $1<j<n-1$, connecting $v_{j-1}$ and $v_{j+2}$ with an edge.  Let $L'$ be the $(2,3)$-assignment for $G'$ obtained by restricting the domain of $L$ to $V(G')$.  By the inductive hypothesis, there is a proportional $L'$-coloring $f$ of $G'$.  Clearly, $|f^{-1}(i)| = a_i - 1$ for $i=1,2$ and $|f^{-1}(3)|=a_3$.  Also, if $1<j<n-1$, $f(v_{j-1}) \neq f(v_{j+2})$.  So, we can find a proportional $L$-coloring of $G$ by coloring the vertices of $G$ that are in $V(G')$ according to $f$ and then coloring $v_j$ and $v_{j+1}$ with 1 and 2 respectively or with 2 and 1 respectively.  So, we have a contradiction.     

Based upon Observation 1, we may assume without loss of generality that $L(v_1) = \{1,2\}$ and $L(v_2) = \{1,3 \}$.

\textbf{Observation 2:} We claim that $L$ assigns the list $\{2,3\}$ to some vertex in $G$.  To see why this is so, suppose that $L$ does not assign the list $\{2,3\}$ to any element in $V(G)$.  By Observation 1 and the fact that all the elements in $\mathcal{L}$ have even multiplicity, we know that $n = 4m$ for some $m \in \N$.  We also know that $L(v_i) = \{1,2 \}$ when $i$ is odd, and $L(v_i) = \{1,3\}$ when $i$ is even.  Now, consider the proper $L$-coloring, $f$, for $G$ given by
\[f(v_i) = \begin{cases} 
      1 & \text{if $i$ is odd and $i \leq 2m$}  \\
      3 & \text{if $i$ is even and $i \leq 2m$} \\
      2 & \text{if $i$ is odd and $i > 2m$} \\
			1 & \text{if $i$ is even and $i > 2m$.}
   \end{cases}
\] 
It is easy to see that $f$ is also a proportional $L$-coloring of $G$ (i.e. $|f^{-1}(1)| = 2m$ and $|f^{-1}(2)|=|f^{-1}(3)|=m$).  So, we have a contradiction. 

Based on Observation 2, we may suppose there is a $t \geq 1$ such that $L$ assigns the list $\{2,3\}$ to the following vertices in $V(G)$: $v_{k_1}, v_{k_2}, \ldots, v_{k_t}$ where $3 \leq k_1 < k_2 < \cdots < k_t \leq n$.

\textbf{Observation 3:} We claim that $L(v_{k_1 - 1})= \{1,2 \}$ (equivalently $k_1$ is even).  To see why this is so, suppose that $L(v_{k_1 - 1})= \{1,3 \}$.  This implies that $k_1$ is odd.  So, for some $l \in \N$ we have that either: (1) $k_1 - 1 = 4l$ or (2) $k_1 - 1 = 4l-2$.  We will derive a contradiction in each case.  For case (1) let $G' = G - \{v_i : i \in [4l] \}$.  Let $L'$ be the $(2,3)$-assignment for $G'$ obtained by restricting the domain of $L$ to $V(G')$.  By the inductive hypothesis, there is a proportional $L'$-coloring $f$ of $G'$.  Clearly, $|f^{-1}(1)|=a_1 - 2l$ and $|f^{-1}(i)| = a_i - l$ for $i=2,3$.  Since $f(v_{4l+1}) = f(v_{k_1}) \neq 1$, we can find a proportional $L$-coloring of $G$ by coloring the vertices of $G$ that are in $V(G')$ according to $f$ and then coloring the vertices in $\{v_i : i \in [4l] \}$ according to $g: \{v_i : i \in [4l] \} \rightarrow [3]$ where 
\[g(v_i) = \begin{cases} 
      1 & \text{if $i$ is odd and $i \leq 2l$}  \\
      3 & \text{if $i$ is even and $i \leq 2l$} \\
      2 & \text{if $i$ is odd and $i > 2l$} \\
			1 & \text{if $i$ is even and $i > 2l$.}
   \end{cases}
\]  
Having constructed a proportional $L$-coloring of $G$, we have reached a contradiction.  For case (2) let $G' = G - \{v_i : i \in [4l-1] \}$.  Let $L'$ be the $(2,3)$-assignment for $G'$ obtained by restricting the domain of $L$ to $V(G')$.  By the inductive hypothesis, there is a proportional $L'$-coloring $f$ of $G'$.  Clearly, $|f^{-1}(1)|=a_1 - (2l-1)$ and $|f^{-1}(i)| = a_i - l$ for $i=2,3$.  Let $h_1: \{v_i : i \in [4l-1] \} \rightarrow [3]$ be given by 
\[h_1(v_i) = \begin{cases} 
      1 & \text{if $i$ is odd and $i \leq 2l-2$}  \\
      3 & \text{if $i$ is even and $i \leq 2l-2$} \\
      2 & \text{if $i$ is odd and $2l-2<i \leq 4l-2$} \\
			1 & \text{if $i$ is even and $2l-2 < i \leq 4l-2$} \\
			3 & \text{if $i = 4l-1$.}
   \end{cases}
\] 
Let $h_2: \{v_i : i \in [4l-1] \} \rightarrow [3]$ be given by 
\[h_2(v_i) = \begin{cases} 
      1 & \text{if $i$ is odd and $i \leq 2l$}  \\
      3 & \text{if $i$ is even and $i \leq 2l$} \\
      2 & \text{if $i$ is odd and $2l<i \leq 4l-1$} \\
			1 & \text{if $i$ is even and $2l < i \leq 4l-1$.} 
   \end{cases}
\] 
It is easy to see that we can construct a proportional $L$-coloring of $G$ by coloring the vertices of $G$ that are in $V(G')$ according to $f$ and then coloring the vertices in $\{v_i : i \in [4l-1] \}$ according to $h_1$ or $h_2$.  Having reached a contradiction in both cases, the proof of Observation 3 is complete.

\textbf{Observation 4:} We claim that $n \geq k_1 + 1$ and $L(v_{k_1 + 1}) = \{1,2 \}$.  We have that $n \geq k_1 + 1$ since $1$ would have odd multiplicity if $n = k_1$.  For the sake of contradiction, suppose that $L(v_{k_1 + 1}) = \{1,3 \}$.  Let $G'$ be the graph obtained from $G$ by deleting the vertices in $\{v_{k_1-2}, v_{k_1-1}, v_{k_1} \}$ and then connecting $v_{k_1-3}$ and $v_{k_1+1}$ with an edge.  Let $L'$ be the $(2,3)$-assignment for $G'$ obtained by restricting the domain of $L$ to $V(G')$.  By the inductive hypothesis, we know there is a proportional $L'$-coloring $f$ of $G'$.  Clearly, $|f^{-1}(i)| = a_i - 1$ for $i \in [3]$.  Let $h$ be the $L$-coloring for $G$ obtained by coloring the vertices of $G$ that are in $V(G')$ according to $f$ and then coloring the vertices $v_{k_1-2}, v_{k_1-1}, v_{k_1}$ with 3, 1, and 2 respectively.  Since $f(v_{k_{1}-3}) \neq 3$, and $f(v_{k_{1}+1}) \neq 2$, we have that $h$ is a proportional $L$-coloring of $G$ which is a contradiction.  

\begin{lem} \label{lem: A}
For each $q \in [t]$, $n \geq k_q+1$.  Moreover, $k_q$ is even and $L(v_{k_q-1}) = L(v_{k_q+1}) = \{1,2\}$.  
\end{lem}

\begin{proof}
Our proof will be by induction on $q$ where $1 \leq q \leq t$.  We have proven the base case in Observations 3 and 4.

So, assume that $1 < q \leq t$ and the desired statement holds for all natural numbers less than $q$.  By the inductive hypothesis, we have that $k_q \geq k_{q-1}+2$.  We begin by showing that $L(v_{k_q-1}) = \{1,2\}$ which by the inductive hypothesis and Observation 1 would immediately imply that $k_q$ is even. 

Suppose for the sake of contradiction that $L(v_{k_q - 1})= \{1,3 \}$.  This implies that $k_q$ is odd and $k_q \geq k_{q-1}+3$.  Since $k_{q-1}$ is even, we know that one of the following two cases holds: (1) $k_q - k_{q-1} - 1 = 4l$ for some $l \in \N$, or (2) $k_q - k_{q-1} - 1 = 4l-2$ for some $l \in \N$.  We will derive a contradiction in each case. For case (1), let $G'$ be the graph obtained from $G$ by deleting the vertices in $\{v_i : k_{q-1}+1 \leq i \leq k_q-1 \}$ and then connecting $v_{k_{q-1}}$ and $v_{k_q}$ with an edge.  Let $L'$ be the $(2,3)$-assignment for $G'$ obtained by restricting the domain of $L$ to $V(G')$.  We know that there is a proportional $L'$-coloring $f$ of $G'$.  Clearly, $|f^{-1}(1)|=a_1 - 2l$ and $|f^{-1}(i)| = a_i - l$ for $i=2,3$.  Since $f(v_{k_{q-1}}) \neq 1$ and $f(v_{k_q}) \neq 1$, we can find a proportional $L$-coloring of $G$ by coloring the vertices of $G$ that are in $V(G')$ according to $f$ and then coloring the vertices in $\{v_i : k_{q-1}+1 \leq i \leq k_q-1 \}$ according to $g: \{v_i : k_{q-1}+1 \leq i \leq k_q-1 \} \rightarrow [3]$ where 
\[g(v_i) = \begin{cases} 
      1 & \text{if $i$ is odd and $k_{q-1}+1 \leq i \leq k_{q-1} + 2l$}  \\
      3 & \text{if $i$ is even and $k_{q-1}+1 \leq i \leq k_{q-1} + 2l$} \\
      2 & \text{if $i$ is odd and $k_{q-1} + 2l <i \leq k_q-1$} \\
			1 & \text{if $i$ is even and $k_{q-1} + 2l <i \leq k_q-1$.}
   \end{cases}
\]  
Having constructed a proportional $L$-coloring of $G$, we have reached a contradiction.  For case (2) let $G'$ be the graph obtained from $G$ by deleting the vertices in $\{v_i : k_{q-1} \leq i \leq k_q-1 \}$ and then connecting $v_{k_{q-1}-1}$ and $v_{k_q}$ with an edge.  Let $L'$ be the $(2,3)$-assignment for $G'$ obtained by restricting the domain of $L$ to $V(G')$.  We know there is a proportional $L'$-coloring $f$ of $G'$.  Clearly, $|f^{-1}(1)|=a_1 - (2l-1)$ and $|f^{-1}(i)| = a_i - l$ for $i=2,3$.  Let $h: \{v_i : k_{q-1} \leq i \leq k_q-1 \} \rightarrow [3]$ be given by 
\[h(v_i) = \begin{cases} 
      1 & \text{if $i$ is odd and $k_{q-1} \leq i \leq k_{q-1}+2l-2$}  \\
      3 & \text{if $i$ is even and $k_{q-1} \leq i \leq k_{q-1}+2l-2$} \\
      2 & \text{if $i$ is odd and $k_{q-1}+ 2l-2<i \leq k_q-1$} \\
			1 & \text{if $i$ is even and $k_{q-1}+ 2l-2 < i \leq k_q-1$.} 
   \end{cases}
\] 
Since $f(v_{k_{q-1}-1}) \neq 3$ and $f(v_{k_q}) \neq 1$, we can construct a proportional $L$-coloring of $G$ by coloring the vertices of $G$ that are in $V(G')$ according to $f$ and then coloring the vertices in $\{v_i : k_{q-1} \leq i \leq k_q-1 \}$ according to $h$.  Having reached a contradiction in both cases, we conclude that $L(v_{k_q-1}) = \{1,2\}$ and $k_q$ is even.

We will now show that $n \geq k_q + 1$ and $L(v_{k_q + 1}) = \{1,2\}$.  First, consider the case where $k_q - k_{q-1} = 2$.  In this case let $\gamma$ be the largest element of $[k_q]$ such that $L(v_{\gamma})=\{1,3 \}$.  We know $\gamma$ is even and exists since $L(v_2)=\{1,3\}$.  Moreover, $k_q - \gamma \geq 4$.  We will prove the desired when: (1) $k_q - \gamma = 4l$ for some $l \in \N$ and (2) $k_q - \gamma = 4l+2$ for some $l \in \N$.  For (1), let $G'$ be the graph obtained from $G$ by deleting the vertices in $\{v_i : \gamma+1 \leq i \leq k_q \}$ and then connecting $v_{\gamma}$ and $v_{k_q+1}$ with an edge (if $v_{k_q+1}$ exists).  Let $L'$ be the $(2,3)$-assignment for $G'$ obtained by restricting the domain of $L$ to $V(G')$.  We know that there is a proportional $L'$-coloring $f$ of $G'$.  Clearly, $|f^{-1}(2)|=a_2 - 2l$ and $|f^{-1}(1)| = a_i - l$ for $i=1,3$.  Let $g:\{v_i : \gamma+1 \leq i \leq k_q \} \rightarrow [3]$ be given by 
\[g(v_i) = \begin{cases} 
      2 & \text{if $i$ is odd and $\gamma+1 \leq i \leq \gamma + 2l$}  \\
      3 & \text{if $i$ is even and $\gamma+1 \leq i \leq \gamma + 2l$} \\
      1 & \text{if $i$ is odd and $\gamma + 2l <i \leq k_q$} \\
			2 & \text{if $i$ is even and $\gamma + 2l <i \leq k_q$.}
   \end{cases}
\]   
Consider the $L$-coloring $h$ of $G$ obtained by coloring the vertices of $G$ that are in $V(G')$ according to $f$ and then coloring the vertices in $\{v_i : \gamma+1 \leq i \leq k_q \}$ according to $g$.  Since $f(v_{\gamma}) \neq 2$ the only way that $h$ is not a proportional $L$-coloring of $G$ is if $v_{k_q+1}$ exists and $f(v_{k_q+1})=2$.  The desired result immediately follows by Observation 1 and the fact that $2 \in L(v_{k_q+1})$. 

For (2) let $G'$ be the graph obtained from $G$ by deleting the vertices in $\{v_i : \gamma \leq i \leq k_q \}$ and then connecting $v_{\gamma-1}$ and $v_{k_q+1}$ with an edge (if $v_{k_q+1}$ exists).  Let $L'$ be the $(2,3)$-assignment for $G'$ obtained by restricting the domain of $L$ to $V(G')$.  We know that there is a proportional $L'$-coloring $f$ of $G'$.  Clearly, $|f^{-1}(2)|=a_2 - (2l+1)$ and $|f^{-1}(1)| = a_i - (l+1)$ for $i=1,3$.  Let $g:\{v_i : \gamma \leq i \leq k_q \} \rightarrow [3]$ be given by 
\[g(v_i) = \begin{cases} 
      3 & \text{if $i$ is even and $\gamma \leq i \leq \gamma+2l$}  \\
      2 & \text{if $i$ is odd and $\gamma \leq i \leq \gamma+2l$} \\
      1 & \text{if $i$ is odd and $\gamma+ 2l < i \leq k_q$} \\
			2 & \text{if $i$ is even and $\gamma+ 2l < i \leq k_q$.} 
   \end{cases}
\] 
Consider the $L$-coloring, $h$, of $G$ obtained by coloring the vertices of $G$ that are in $V(G')$ according to $f$ and then coloring the vertices in $\{v_i : \gamma \leq i \leq k_q \}$ according to $g$.  Since $f(v_{\gamma-1}) \neq 3$ the only way that $h$ is not a proportional $L$-coloring of $G$ is if $v_{k_q+1}$ exists and $f(v_{k_q+1})=2$.  The desired result immediately follows by Observation 1 and the fact that $2 \in L(v_{k_q+1})$.  

Finally, suppose that $k_q - k_{q-1} > 2$.  In this case we know $v_{k_q-2} = \{1,3\}$.  Let $G'$ be the graph obtained from $G$ by deleting the vertices in $\{v_{k_q-2}, v_{k_q-1}, v_{k_q} \}$ and then connecting $v_{k_{q}-3}$ and $v_{k_q+1}$ with an edge (if $v_{k_q+1}$ exists).  Let $L'$ be the $(2,3)$-assignment for $G'$ obtained by restricting the domain of $L$ to $V(G')$.  We know there is a proportional $L'$-coloring $f$ of $G'$.  Clearly, $|f^{-1}(i)| = a_i - 1$ for $i \in [3]$.  Let $h_1$ be the $L$-coloring for $G$ obtained by coloring the vertices of $G$ that are in $V(G')$ according to $f$ and then coloring the vertices $v_{k_q-2}, v_{k_q-1}, v_{k_q}$ with 1, 2, and 3 respectively.  Let $h_2$ be the $L$-coloring for $G$ obtained by coloring the vertices of $G$ that are in $V(G')$ according to $f$ and then coloring the vertices $v_{k_q-2}, v_{k_q-1}, v_{k_q}$ with 3, 1, and 2 respectively.  We know that neither $h_1$ nor $h_2$ is a proportional $L$-coloring of $G$. Notice that the only way that both $h_1$ and $h_2$ are not proportional $L$-colorings of $G$ is if $v_{k_q+1}$ exists, $f(v_{k_{q}-3})= 1$, and $f(v_{k_{q}+1})=2$.  So, $v_{k_q+1}$ exists (i.e. $n \geq k_q + 1$), and $2 \in L(v_{k_q+1})$.  Observation 1 then implies that $L(v_{k_q + 1}) = \{1,2\}$.

This completes the induction step and our proof is complete.       
\end{proof}

\begin{lem} \label{lem: R}
 For any $i \in [n]$, $L(v_i) = \{1,2\}$ if and only if $i$ is odd. Consequently, $L(v_i) \neq \{1,2\}$ if and only if $i$ is even.  
\end{lem}

\begin{proof}
Suppose that $i \in [n]$ and $i$ is odd.  For the sake of contradiction suppose that $L(v_i) \neq \{1,2\}$.  Lemma~\ref{lem: A} implies that for each $q \in [t]$, $i \neq k_q$.  So, $L(v_i) = \{1,3\}$.  Let $v_{k_a}$ be the element of $\{v_{k_q} : q \in [t]\}$ that is closest in distance to $v_i$.  Without loss of generality, suppose that $k_a > i$.  Since $i$ is odd and $k_a$ is even, Observation 1 implies that $L(v_{k_a - 1})= \{1,3\}$ which contradicts Lemma~\ref{lem: A}.

Conversely, suppose $L(v_i) = \{1,2\}$, and for the sake of contradiction $i$ is even.  Let $v_{k_a}$ be the element of $\{v_{k_q} : q \in [t]\}$ that is closest in distance to $v_i$.  Without loss of generality, suppose that $k_a > i$.  Since $i$ is even and $k_a$ is even, Observation 1 implies that $L(v_{k_a - 1})= \{1,3\}$ which contradicts Lemma~\ref{lem: A}. 
\end{proof}  

\begin{lem} \label{lem: B}
$L(v_n) = \{1,2\}.$  Consequently, $n$ is odd.
\end{lem} 

\begin{proof} 
For the sake of contradiction, suppose that $L(v_n) \neq \{1,2\}$.  Lemma~\ref{lem: A} then implies that $L(v_n) = \{1,3 \}$.  This means that $n$ is even by Lemma~\ref{lem: R}.  We will derive a contradiction in the following cases: (1) $n - k_t = 4l$ for some $l \in \N$, and (2) $n - k_t = 4l-2$ for some $l \in \N$. For case (1), let $G' = G - \{v_i : k_{t}+1 \leq i \leq n \}$.  Let $L'$ be the $(2,3)$-assignment for $G'$ obtained by restricting the domain of $L$ to $V(G')$.  We know that there is a proportional $L'$-coloring $f$ of $G'$.  Clearly, $|f^{-1}(1)|=a_1 - 2l$ and $|f^{-1}(i)| = a_i - l$ for $i=2,3$.  Since $f(v_{k_{t}}) \neq 1$, we can find a proportional $L$-coloring of $G$ by coloring the vertices of $G$ that are in $V(G')$ according to $f$ and then coloring the vertices in $\{v_i : k_{t}+1 \leq i \leq n \}$ according to $g: \{v_i : k_{t}+1 \leq i \leq n \} \rightarrow [3]$ where 
\[g(v_i) = \begin{cases} 
      1 & \text{if $i$ is odd and $k_{t}+1 \leq i \leq k_{t} + 2l$}  \\
      3 & \text{if $i$ is even and $k_{t}+1 \leq i \leq k_{t} + 2l$} \\
      2 & \text{if $i$ is odd and $k_{t} + 2l <i \leq n$} \\
			1 & \text{if $i$ is even and $k_{t} + 2l <i \leq n$.}
   \end{cases}
\]  
Having constructed a proportional $L$-coloring of $G$, we have reached a contradiction.  For case (2) let $G' = G - \{v_i : k_{t} \leq i \leq n \}$.  Let $L'$ be the $(2,3)$-assignment for $G'$ obtained by restricting the domain of $L$ to $V(G')$.  We know there is a proportional $L'$-coloring $f$ of $G'$.  Clearly, $|f^{-1}(1)|=a_1 - (2l-1)$ and $|f^{-1}(i)| = a_i - l$ for $i=2,3$.  Let $h: \{v_i : k_{t} \leq i \leq n \} \rightarrow [3]$ be given by 
\[h(v_i) = \begin{cases} 
      1 & \text{if $i$ is odd and $k_t \leq i \leq k_{t}+2l-2$}  \\
      3 & \text{if $i$ is even and $k_t \leq i \leq k_t+2l-2$} \\
      2 & \text{if $i$ is odd and $k_t+ 2l-2<i \leq n$} \\
			1 & \text{if $i$ is even and $k_t+ 2l-2 < i \leq n$.} 
   \end{cases}
\] 
It is easy to see that since $f(v_{k_t - 1}) \neq 3$, we can construct a proportional $L$-coloring of $G$ by coloring the vertices of $G$ that are in $V(G')$ according to $f$ and then coloring the vertices in $\{v_i : k_{t} \leq i \leq n \}$ according to $h$.  Having reached a contradiction in all cases, we conclude that $L(v_n) = \{1,2\}$.  It immediately follows that $n$ is odd as well.
\end{proof}

We now present two definitions.  In what follows, suppose $H = P_m$ where $m$ is an even natural number.  Suppose the vertices of $H$ in order are $w_1, w_2, \ldots, w_m$.  Suppose also that $K$ is the $(2,3)$-assignment for $H$ given by: $L(w_m) = \{2,3 \}$, $L(w_i) = \{1,2\}$ when $i$ is odd, and $L(w_i) = \{1,3 \}$ when $i$ is even and less than $m$.  We say $f: V(H) \rightarrow [3]$ is an \emph{$\alpha_2$-coloring} of $H$ if 
\[f(w_i) = \begin{cases} 
      1 & \text{if $i$ is odd}  \\
      3 & \text{if $i$ is even and $i <m$} \\
      2 & \text{if $i = m$.} 
   \end{cases}
\] 
Notice an $\alpha_2$-coloring of $H$ is a proper $K$-coloring of $H$ that uses: one $m/2$ times, three $(m/2-1)$ times, and two once.  We say $f: V(H) \rightarrow [3]$ is an \emph{$\alpha_3$-coloring} of $H$ if 
\[f(w_i) = \begin{cases} 
      2 & \text{if $i$ is odd}  \\
      1 & \text{if $i$ is even and $i <m$} \\
      3 & \text{if $i = m$.} 
   \end{cases}
\] 
Notice an $\alpha_3$-coloring of $H$ is a proper $K$-coloring of $H$ that uses: one $(m/2-1)$ times, two $m/2$ times, and three once.   

One key property to notice is that if $D$ is the disjoint union of two copies of $P_m$ and $R$ is the $(2,3)$-assignment for $D$ obtained by using the rule for $K$ on each component of $D$, then if we use an $\alpha_2$-coloring to color one component of $D$ and an $\alpha_3$-coloring to color the other component of $D$, we obtain a proportional $R$-coloring of $D$. 

We are now ready to finish the proof.  Let $z$ be the smallest element in $[(n-1)/2]$ with the property $L(v_z) \neq L(v_{n-z+1})$.  Notice that $z$ must exist, for if $L(v_i) = L(v_{n-i+1})$ for each $i \in [(n-1)/2]$, then the colors in $L(v_{(n+1)/2})$ would have odd multiplicity.  Since $z$ and $n-z+1$ have the same parity, we know that $z$ is even.  We may also assume without loss of generality that $L(v_z) = \{2,3\}$ and $L(v_{n-z+1}) = \{1,3\}$.  This means that $z = k_q$ for some $q \in [t]$.  Let 
$$G' = G - \left (\{v_i : i \in [z-1] \} \bigcup \{v_{n-i+1} : i \in [z-1] \} \right ).$$  
Let $L'$ be the $(2,3)$-assignment for $G'$ obtained by restricting the domain of $L$ to $V(G')$.  By the definition of $z$, note $\eta_{L'}(i)$ is even for each $i \in [3]$.  By the inductive hypothesis, there is a proportional $L'$-coloring $f$ of $G'$.  We now describe how to extend $f$ to a proportional $L$-coloring of $G$ which will give us a contradiction and complete the proof.

Let $k_0=0$.  If $q \geq 2$, for each $j \in [q-1]$ consider the vertices in the sets: $A_j = \{v_i : k_{j-1} + 1 \leq i \leq k_j \}$ and $B_j = \{v_{n-i+1} : k_{j-1} + 1 \leq i \leq k_j \}$.  Consider $G[A_j]$, letting $w_s = v_{k_{j-1} + s}$ for each $s \in [k_j - k_{j-1}]$, color $G[A_j]$ with the $\alpha_2$-coloring of $G[A_j]$.  Consider $G[B_j]$, letting $w_s = v_{n - k_{j-1} + 1 - s}$ for each $s \in [k_j - k_{j-1}]$, color $G[B_j]$ with the $\alpha_3$-coloring of $G[B_j]$.  After doing this for each $j \in [q-1]$ our coloring is still proper since an $\alpha_2$-coloring starts with 1 and ends with 2, and an $\alpha_3$-coloring starts with 2 and ends with 3.

Finally, consider the vertices in the sets: $A_q = \{v_i : k_{q-1} + 1 \leq i \leq k_q - 1 \}$ and $B_q = \{v_{n-i+1} : k_{q-1} + 1 \leq i \leq k_q - 1 \}$.  Let $h_1 : A_q \rightarrow [3]$ be given by 
\[h_1(v_i) = \begin{cases} 
      1 & \text{if $i$ is odd}  \\
      3 & \text{if $i$ is even}  
   \end{cases}
\]   
Let $h_2 : B_q \rightarrow [3]$ be given by 
\[h_2(v_i) = \begin{cases} 
      2 & \text{if $i$ is odd}  \\
      1 & \text{if $i$ is even}  
   \end{cases}
\]   
Since we colored $v_{k_{q-1}}$ with 2 (if $q \geq 2$), $f(v_z) \neq 1$, we colored $v_{n - k_{q-1} + 1}$ with 3 (if $q \geq 2$), and $f(v_{n-z+1}) \neq 2$, coloring $G[A_q]$ and $G[B_q]$ according to $h_1$ and $h_2$ respectively completes a proportional $L$-coloring of $G$.  Our proof of Proposition~\ref{pro: evenmult} is now complete.
\end{proof}

\begin{lem} \label{lem: finpath} 
Suppose $n \in \N$.  If $G = P_n$,  then $G$ is proportionally $(2,3)$-choosable.
\end{lem}

\begin{proof}
Suppose that $L$ is an arbitrary $(2,3)$-assignment for $G$.  We must show that $G$ is proportionally $L$-colorable.  In the case that the multiplicity of each color in $\mathcal{L}$ is even, the result is implied by Proposition~\ref{pro: evenmult}.

So, assume without loss of generality that $\eta_L(1)$ and $\eta_L(2)$ are odd.  Let $G'$ be the path obtained from $G$ by adding a new vertex $w$ and connecting $w$ to an endpoint of $G$ with an edge.  Let $L'$ be the $(2,3)$-assignment for $G'$ given by $L'(v) = L(v)$ for each $v \in V(G)$ and $L(w) = \{1,2 \}$.  By Proposition~\ref{pro: evenmult}, there is a proportional $L'$-coloring $f$ of $G'$.  Restricting the domain of $f$ to $V(G)$ yields a proportional $L$-coloring of $G$.   
\end{proof}

We are now ready to prove Theorem~\ref{thm: lis3}.

\begin{proof}
Suppose a connected graph $G$ is proportionally $(2,3)$-choosable. By Lemmas~\ref{pro: degree} and~\ref{pro: cycle is G}, we have $G$ is a path.  Conversely, suppose $G = P_n$. By Lemma~\ref{lem: finpath}, $G$ must be proportionally $(2,3)$-choosable.
\end{proof}

Finally we prove Proposition~\ref{pro: lis3} to illustrate the difficulty with giving a complete characterization of the proportionally $(2,3)$-choosable graphs.

\begin{custompro}{\bf \ref{pro: lis3}}
$C_4 + P_1$ is proportionally $(2,3)$-choosable.
\end{custompro}

\begin{proof}
Let $G = C_4 + P_1$, where the vertices of the copy of $C_4$ used to form $G$ written in cyclic order are: $v_1, v_2, v_3, v_4$, and the vertex set of the copy of $P_1$ used to form $G$ is $\{u_1\}$. For the sake of contradiction, assume $G$ is not proportionally $(2,3)$-choosable. Suppose $L$ is a $(2,3)$-assignment for $G$ for which $G$ is not proportionally $L$-colorable, and suppose $\eta_L(i) = a_i$ for each $i \in [3]$.

We first claim that $L(v_i) \neq L(v_{i+1})$ for each $i \in [3]$ and $L(v_1) \neq L(v_4)$. To see why, without loss of generality suppose for the sake of contradiction that $L(v_1) = L(v_2) = \{1, 2\}$. Let $G' = G - \{v_1, v_2\}$, and let $L'$ be the $(2,3)$-assignment for $G'$ obtained by restricting the domain of $L$ to $V(G')$. Since $G'$ is a copy of $P_2 + P_1$, there exists a proportional $L'$-coloring $f$ of $G'$ by Proposition~\ref{pro: span} and Theorem~\ref{thm: lis3}. Notice that $\lfloor a_i/2 \rfloor - 1 \leq |f^{-1}(i)| \leq \lceil a_i/2 \rceil - 1$ for $i = 1,2$, and $\lfloor a_3/2 \rfloor \leq |f^{-1}(3)| \leq \lceil a_3/2 \rceil$. Furthermore, since $f(v_3) \neq f(v_4)$, we can find a proportional $L$-coloring of $G$ by coloring the vertices in $V(G')$ according to $f$, then coloring $v_1$ and $v_2$ with $1$ and $2$ respectively or with $2$ and $1$ respectively; hence, we have a contradiction.

We may now assume without loss of generality that $L(v_1) = \{1, 2\}$ and $L(v_2) = \{1, 3\}$. We claim that $L(v_3) = \{1, 2\}$. To see why, suppose for the sake of contradiction that $L(v_3) \neq \{1, 2\}$. By the most recent argument, we have that $L(v_3) = \{2, 3\}$. Let $G' = G - \{v_1, v_2, v_3\}$, and let $L'$ be the $(2,3)$-assignment for $G'$ obtained by restricting the domain of $L$ to $V(G')$. Since $G'$ is a copy of $P_1+P_1$, there exists a proportional $L'$-coloring $f$ of $G'$ by Proposition~\ref{pro: span} and Theorem~\ref{thm: lis3}. Notice that $\lfloor a_i/2 \rfloor - 1 \leq |f^{-1}(i)| \leq \lceil a_i/2 \rceil - 1$ for $i \in [3]$. If $f(v_4) \neq 2$, we can find a proportional $L$-coloring of $G$ by coloring the vertices in $V(G')$ according to $f$, then coloring $v_1, v_2, v_3$ with $1, 3, 2$ respectively or with $2, 1, 3$ respectively. So, suppose $f(v_4) = 2$. This means $L'(v_4) = \{a, 2\}$ for some $a \in \{1, 3\}$. We claim there must be an $L'$-coloring $f'$ of $G'$ such that $f'(v_4) \neq 2$. If $2 \in L'(u_1)$, then we can define $f'$ such that $f'(v_4) = a$ and $f'(u_1) = 2$. If $2 \not\in L'(u_1)$, then we know that $L'(u_1) = \{a, c\} = \{1,3\}$. So, we can define a proportional $L'$-coloring $f'$ of $G'$ such that $f'(v_4) = a$ and $f'(u_1) = c$. We can then find a proportional $L$-coloring of $G$ by coloring the vertices in $V(G')$ according to $f'$, then coloring $v_1, v_2, v_3$ with $1,3,2$ respectively or with $2,1,3$ respectively; hence, we have a contradiction.

By a similar argument, one can deduce that $L(v_4) = \{1, 3\}$.

Finally, we have three cases for $L(u_1)$: (1) $L(u_1) = \{1, 2\}$, (2) $L(u_1) = \{1, 3\}$, or (3) $L(u_1) = \{2, 3\}$. For case (1), we can find a proportional $L$-coloring of $G$ by coloring $v_1, v_2, v_3, v_4, u_1$ with $2,1,2,3,1$ respectively. For case (2), we can find a proportional $L$-coloring of $G$ by coloring $v_1, v_2, v_3, v_4, u_1$ with $1,3,2,3,1$ respectively. For case (3), we can find a proportional $L$-coloring of $G$ by coloring $v_1, v_2, v_3, v_4, u_1$ with $1,3,1,3,2$ respectively. Having reached a contradiction in each case, we have that $G$ is proportionally $(2,3)$-choosable.
\end{proof}

{\bf Acknowledgment.}  The authors would like to thank Hemanshu Kaul, Michael Pelsmajer, and Jonathan Sprague for their helpful comments on this paper.  The authors would also like to thank Carlos Villeda for many helpful conversations.

\appendix

\section{Appendix}

Suppose $G=P_7$.  For each possible $(2,4)$-assignment, $L$, for $G$ the following program determines whether there exists a proportional $L$-coloring of $G$. If there is a $(2,4)$-assignment, $L$, for $G$ for which there is no proportional $L$-coloring, then the function \texttt{bad\_2\_4\_assignment\_P7()} returns $L$. Otherwise, the function returns an empty list, and $G$ is proportionally $(2,4)$-choosable.
Since the output of \texttt{bad\_2\_4\_assignment\_P7()} is an empty list, $P_7$ is proportionally $(2,4)$-choosable.

\begin{verbatim}
import math

def proportional_exists_recursive(L, arr, color_counts, usage_bounds):
    # Base Case
    if len(arr) == len(L):
        for i in range(1, len(color_counts)):
            if color_counts[i] < usage_bounds[0][i]: return False
        return True
    
    # Recursive Case
    found = False
    for color in L[len(arr)]:
        # Check if coloring is found, improper, or not proportional
        if found: break
        if color == arr[-1]: continue
        if color_counts[color] + 1 > usage_bounds[1][color]: continue
        
        color_counts[color] += 1
        found = proportional_exists_recursive(L, arr + [color],
        color_counts, usage_bounds)
        color_counts[color] -= 1
    
    if found: return True
    return False

def proportional_list_coloring_exists_P(L, m, k):
    usage_bounds = [[0]*(k + 1), [0]*(k + 1)] # Bounds on multiplicity
    
    color_frequencies = [0]*(k + 1)
    for L_v in L: # Find how often each color shows up
        for color in L_v:
            color_frequencies[color] += 1
    
    for i in range(1, len(color_frequencies)): # Determine bounds
        usage_bounds[0][i] = math.floor(color_frequencies[i] / m)
        usage_bounds[1][i] = math.ceil(color_frequencies[i] / m)
    
    found = False
    for color in L[0]: # Begin search for proportional L-coloring
        if found: break
        color_counts = [0]*(k + 1)
        color_counts[color] = 1
        found = proportional_exists_recursive(L, [color],
        color_counts, usage_bounds)
    
    if found: return True # Proportional L-coloring exists
    return False # Does not exist

def bad_2_4_assignment_P7():
    possible_lists = [[1, 2], [1, 3], [1, 4], [2, 3], [2, 4], [3, 4]]
    num_possible_lists = len(possible_lists)
    num_assignments = num_possible_lists ** 7
    L = [0]*7 # List Assignment
    
    for i in range(num_assignments): # Iterate through all assignments
        temp = i
        for j in range(7):
            pv = num_possible_lists ** (6 - j)
            L[j] = possible_lists[temp // pv]
            temp = temp % pv
        if not proportional_list_coloring_exists_P(L, 2, 4):
            return L # Bad assignment
    return [] # All (2,4)-assignments L have a proportional L-coloring

if __name__ == "__main__":
    print(bad_2_4_assignment_P7())
\end{verbatim}

\end{document}